\documentclass[12pt,a4paper]{article}
\usepackage{latexsym,epsfig}
\usepackage{amsmath,amssymb,amsthm}
\usepackage{epstopdf}
\usepackage{subfigure}

\newtheorem{theorem}{Theorem}

\newtheorem{lemma}{Lemma}
\newtheorem{proposition}[theorem]{Proposition}
\newtheorem{corollary}[theorem]{Corollary}
\newtheorem{observation}{Observation}

\theoremstyle{definition}
\newtheorem{example}{Example}
\newtheorem{definition}{Definition}

\newcommand\lrceil[2]{\lceil #1 / #2 \rceil}
\newcommand{\A}{\mathcal A}

\mathcode`l="8000
\begingroup
\makeatletter
\lccode`\~=`\l
\DeclareMathSymbol{\lsb@l}{\mathalpha}{letters}{`l}
\lowercase{\gdef~{\ifnum\the\mathgroup=\m@ne \ell \else \lsb@l \fi}}%
\endgroup

\title{Online and quasi-online colorings of wedges and intervals}

\date{}
\author{
Bal\'azs Keszegh\thanks{Alfr\'ed R\'enyi Institute of Mathematics,
  Hungarian Academy of Sciences, Budapest, Hungary. Email:
\texttt{keszegh.balazs@renyi.mta.hu}. 
Research supported by Hungarian National Science Fund (OTKA), under grant PD 108406, NN 102029 (EUROGIGA project GraDR 10-EuroGIGA-OP-003) and by the J\'anos Bolyai Research Scholarship of the Hungarian Academy of Sciences.}
\and Nathan Lemons\thanks{Los Alamos National Laboratory, Theoretical Division. Email: \texttt{nlemons@lanl.gov}.} \and D\"om\"ot\"or P\'alv\"olgyi\thanks{Department of Computer Science, E\"otv\"os University. Email:\texttt{dom@cs.elte.hu}. Research supported by Hungarian National Science Fund (OTKA), under grant PD 104386, NN 102029 (EUROGIGA project GraDR 10-EuroGIGA-OP-003) and by the J\'anos Bolyai Research Scholarship of the Hungarian Academy of Sciences.}
}
\begin{document}
\maketitle

%\begin{abstract}
%\end{abstract}

\abstract
We consider proper online colorings of hypergraphs defined by geometric regions.  We prove that
there is an online coloring algorithm that colors $N$ intervals of the real line using $\Theta(\log
N/k)$ colors such that for every point $p$, contained in at least $k$ intervals, not all the
intervals containing $p$ have the same color. We also prove the corresponding result about online
coloring a family of wedges (quadrants) in the plane that are the translates of a given fixed wedge. These results contrast the results of the first and
third author showing that in the quasi-online setting $12$ colors are enough to color wedges
(independent of $N$ and $k$). We also consider quasi-online coloring of intervals. In all cases we present efficient coloring algorithms.

\section{Introduction}\label{intro}

The study of proper colorings of geometric hypergraphs has attracted much attention, not only
because this is a very basic and natural theoretical problem but also because such problems have
important applications.  One such application area is resource allocation: to determine the number
of CPUs necessary to run several jobs, each with fixed starting and stopping times is exactly the
problem of finding the chromatic number of the associated interval graph.  Similarly, the coloring
of geometric shapes in the plane is related to the problems of cover decomposability and conflict
free colorings; these problems have applications in sensor networks and frequency assignment as well
as other areas.  For surveys on these and related problems see Refs.~\cite{domisurvey,cdsurvey,cfsurvey,trottersurvey}.

Despite the well-known applications of colorings of geometric graphs and hypergraphs, relatively
little attention has been paid to the online and quasi-online versions of these problems.  Online
and quasi online coloring problems are natural to consider from both a theoretical point of view (as
a means to better understand and relate various geometrical hypergraphs) as well as from a practical
point of view: many of the natural applications require streaming/online algorithms.  

Before describing the contributions of this paper, we formally define the hypergraphs and colorings
under consideration.

%\subsection{Geometric hypergraphs}
\begin{definition}\label{def:geometric_objects}
\hfill
\begin{description}
\item[Wedge:] the set of points $\{(x,y)\in \mathbb{R}\times \mathbb{R} \mid x< x_0;
y< y_0 \}$ for a fixed $(x_0, y_0)$ called the apex.\footnote{Here, and similarly at the later definitions as well, we also allow $x_0=\infty$ and $y_0=\infty$.}
\item[Octant:] the set of points $\{(x,y,z)\in \mathbb{R}\times \mathbb{R}\times \mathbb{R}\mid
x< x_0; y< y_0;z< z_0 \}$ for fixed $x_0, y_0, z_0$.
\item[Bottomless rectangle:] the set of points $\{(x,y)\in \mathbb{R}\times \mathbb{R}\mid x_0< x<
x_1; y< y_0\}$ for fixed $x_0, x_1, y_0$.
\item[Interval:] the set of points $\{x\in \mathbb{R}\mid x_0< x< x_1\}$ for fixed
$x_0, x_1$.
\item[Diagonal line:] the set of points $\{x\in \mathbb{R}\times \mathbb{R}\mid x+y=c\}$ for fixed
$c$.
%the line on the plane defined by the function $y=-x$.
\end{description}
\end{definition}

Each of the above geometric objects defines a natural class of objects: for example the class of
wedges in $\mathbb{R}^2$ or the class of intervals in $\mathbb{R}$.
We consider hypergraphs which can be induced through these classes of geometrical objects.

Let $S$ be a set and let $\mathcal{O}$ be a family of its subsets.
For any finite subset $X$ of $S$, the primal hypergraph induced by $X$ and $O$ is the following. Its vertices correspond to the points in $X$ and its hyperedges correspond to those subsets of $X$ that can be obtained as the intersection of $X$ with a member of $\mathcal {O}$. More precisely:
  
\begin{definition}[Primal Hypergraph Construction]
For a base set $S$ and family $\mathcal {O}$ of subsets of $S$, the points $x_1,x_2,\ldots,x_n\in S$ induce, with respect to $\mathcal{O}$, a \emph{primal hypergraph}, $H$, on $n$ vertices $v_1,\ldots, v_n$ where for each $I\subset 2^{[n]}$
$$e_I=\{v_{i}:i\in I\} \mbox{ is a hyperedge of $H$ iff } \exists O\in\mathcal{O}:\;O\cap\{x_1,x_2,\ldots,x_n\} = \{x_i:i\in I\}.$$
\end{definition}

%When it leads to no confusion,
For geometric objects, $S$ is the space in which the objects are defined in Definition~\ref{def:geometric_objects},
%we will use the respective spaces given in Definition \ref{def:geometric_objects} (
e.g., for ``wedges'' %will always be assumed to live in the space
it is $\mathbb{R}^2$. 

\begin{example}
Let $S=\mathbb{R}^2$  and let $\mathcal{O}$ be the collection of all wedges.
Then the points $(0,0),(1,0)$ and $(0,1)$ induce the primal hypergraph $H$ consisting
of the hyperedges $\{v_1,v_2,v_3\}$, $\{v_1,v_2\}$, $\{v_1,v_3\}$, $\{v_1\}$, and $\emptyset$.
\end{example} 
 
There is a second, dual way to create a hypergraph from a set system. Let $S$ be a set and let $\mathcal{O}$ be a family of its subsets.
For any $\mathcal{O}'$ finite subfamily of $\mathcal{O}$, the dual hypergraph induced by $\mathcal{O}'$ with respect to $S$ is the following. Its vertices correspond to the sets in $\mathcal{O}'$ and its hyperedges correspond to those subfamilies of $\mathcal{O}'$ that can be obtained as the subfamily of sets in $\mathcal{O}'$ that contain a point $x$ of $S$. More precisely:

\begin{definition}[Dual Hypergraph Construction]
For a base set $S$ and family $\mathcal {O}$ of subsets of $S$, the objects $O_1,O_2,\ldots,O_n\in \mathcal{O}$ induce with respect to $\mathcal{S}$ a \emph{dual hypergraph}, $H$, on $n$ vertices $v_1,\ldots, v_n$ where for each $I\subset 2^{[n]}$
$$e_I=\{v_{i}:i\in I\} \mbox{ is a hyperedge of $H$ iff } \exists x\in S:\;\{i:x\in O_i, 1\leq i \leq
n\} = I.$$
\end{definition}

In general, for a fixed geometric space $S$ and set of objects $\mathcal{O}$, the class of
hypergraphs which can be formed through the primal construction is not the same as the class of
hypergraphs which can be formed through the dual construction. However, as observed in Pach \cite{P80}, in some important cases, these two classes are actually the same.  

\begin{proposition}
Let $\mathcal{O}$ be the family of the translates of some fixed Euclidean geometric set, e.g., a wedge $W$.
If $H$ is a hypergraph induced (through the primal construction) by the points $x_1,x_2,\ldots,x_n$, then there exist $O_1,O_2,\ldots,O_n\in \mathcal{O}$ which also induce $H$ through the dual construction.
Similarly if $H$ is a geometric hypergraph induced (through the dual construction) by the wedges $O_1,O_2,\ldots,O_n\in \mathcal{O}$, then there exist points $x_1,x_2,\ldots,x_n$ which also induce $H$ through the primal construction.
\end{proposition} 
\begin{proof}
Fix a point $c\in W$ which we call the {\em center} of $W$ and we say that $W$ is {\em centered} at $c$.
Denote the centrally reflected translates of $W$ by $\bar W$ and call the reflection of the center of $W$ the center of $\bar W$.
Consider the operation $\Psi$ that takes each wedge to its center and each point $c$ to a reflected translate $\bar W$ centered on $c$.
By definition, $\Psi$ preserves point-object incidences.
%The objects of the dual construction will be $\bar W$ centered at $x_1,x_2,\ldots,x_n$.
As the family of translates of $W$ and $\bar W$ induce the same hypergraphs, we have proved the equivalence.
%For some translate $W'$ we have $x_i\in W'$ if and only if $\bar W$ centered at $x_i$ contains the center of $W'$.
\end{proof} 

\begin{definition}
Given a finite hypergraph $H$, a (partial) coloring of the vertices of $H$ is a {\em $k$-proper (partial) $c$-coloring} if it uses $c$ colors and no hyperedge of size at least $k$ is monochromatic.
\end{definition}

%In particular, we are interested in finding algorithms that $k$-proper $c$-color various classes of geometric hypergraphs in an online or quasi-online way (to be defined later).
When %the parameters $k$ and $c$ are
obvious from the context, we may refer to a $k$-proper (partial) $c$-coloring simply as a {\em proper coloring}.
We will consider the proper coloring problem for both geometric hypergraphs induced by the primal as well as the dual constructions.
To simplify the exposition, we will avoid referring to the hypergraph $H$ explicitly.
Rather we will speak of coloring points with respect to objects (primal construction) or of coloring objects with respect to points (dual construction).
In particular, if the points/vertices are colored in a primal construction we say that an object is monochromatic if the corresponding hyperedge is monochromatic. The {\em size of a geometric object} (e.g., size of a wedge) will refer to the number of points in the geometric object in the primal construction, and the {\em depth of a point} will refer to the number of objects containing a point in the dual construction.

%\subsection{Online colorings}
In online coloring problems, the set of objects to be colored is not known beforehand; objects come to be colored one-by-one and a proper coloring must be maintained at all times. %In quasi-online coloring, objects come online and must be colored one by one, such that a valid coloring is maintained at each step, yet the objects and their order are known in advance.
This problem has several variants, below we give an exact definition of the types interesting to us.
To emphasize the difference, we refer to proper colorings as {\em offline colorings}.
%KESZEGH: THE NEXT DEFINITIONS ARE ONLY FOR THE DUAL SETTING, SHOULD NOT THEY BE A BIT MORE GENERAL TO COVER THE PRIMAL SETTING DIRECTLY TOO?
%DOM: WHY DON'T WE TALK ABOUT HYPEREDGES/VERTICES INSTEAD OF OBJECTS?

\begin{definition}
Let $H$ be a hypergraph on $n$ vertices and let $v_1,v_2,\ldots,v_n$ be an ordering of the vertices.
For each $i$ let $H_i$ be the hypergraph on the vertices $V_i=\{v_1,v_2,\ldots,v_i\}$ with edges
$\{e|_{V_i},\;e\in E(H)\}$.
A $k$-proper $c$-coloring algorithm $\A$ of $H$ for which each $H_i$ is $k$-properly partially
$c$-colored is called 
\begin{description}
%\item[offline:] $\A$ properly colors the objects knowing $\mathcal{O}$; %the coloring of $\mathcal{O}$ has to satisfy $P$;
\item[online] if at the beginning $\A$ knows nothing about $H$, in step $i$, $H_i$ is presented to
$\A$ and $\A$ colors $v_i$ (the vertices $v_j$ for $j<i$ retain their colors from the previous steps).
Note that $\A$ knows nothing about the future vertices $v_j, j>i$;
\item[semi-online] if at the beginning $\A$ knows nothing about $H$, in step $i$ $H_i$ is presented
to $\A$ and $\A$ colors some (maybe zero) of the as yet uncolored vertices (the vertices $v_j$ for $j<i$
retain their colors from the previous steps).  Again $\A$ knows nothing about the
vertices that come later;
\item[quasi-online\footnotemark]\footnotetext{Quasi-online colorings are also known as colorings of
{\em dynamic} point sets \cite{A13} or as colorings of ordered point sets \cite{cd,multicd}. We
shall use the quasi-online notation to emphasize that it lies between the offline and online
coloring models.} if $\A$ colors the vertices knowing the full hypergraph $H$.
\end{description}
\end{definition}

By definition the above types are ordered by hardness in the way they are presented.

\begin{observation}
Every online algorithm is also a semi-online algorithm. Every semi-online algorithm is also a quasi-online algorithm. Every quasi-online algorithm is also an offline algorithm.
\end{observation}

%As we can see, to complicate matters further, there is another similar, moreover, similarly denoted model in the literature, called semi-online coloring \cite{A13}.
Usually semi-online algorithms are presented when a quasi-online algorithm is needed (e.g., for coloring points with respect to intervals \cite{A13}), yet in other cases this is not a possible route as a quasi-online algorithm exists while a semi-online algorithm does not (e.g., for coloring wedges \cite{cd,colorful2}).
In this paper we consider quasi-online and online colorings of wedges and intervals.

%\subsection{Background and related work}
One major motivation to study quasi-online colorings is that it can be used to solve corresponding
offline higher dimensional problems. In particular, it was shown that octants can be offline colored
using two colors such that there is no monochromatic octant of size at least $12$ \cite{cd}. Indeed
by projecting the octants on the $xy$ plane and ordering them by the $z$-coordinates of their
apices, it was shown that quasi-online coloring the resulting wedges is equivalent to offline coloring the original octants \cite{cd}. 

Knowing that one can properly color wedges quasi-online using a constant number of colors, motivated Tardos \cite{tardos} to ask whether a proper coloring can be achieved in the online setting also, possibly with a larger $k$ and more colors. It is easy to see that $2$ colors are not enough to guarantee non-monochromatic wedges (i.e., there may be arbitrarily large monochromatic wedges), even when the points are restricted to a diagonal line. While it is possible to $2$-properly $3$-color if the points are restricted to a diagonal line \cite{cd}, for general point sets the answer turns out to be more complicated. 

Answering the question of Tardos, Theorem \ref{thm:noonline} shows that in general no finite number
of colors are enough.
Formally, for any $c$ and $k$, and any online $c$-coloring algorithm, there exists a finite set of points in the plane for which the algorithm produces a monochromatic wedge of size at least $k$. This implies that the same holds for coloring intervals with respect to points, that is, there is no online algorithm that $k$-properly $c$-colors intervals with respect to points. 
%KESZEGH: ALTHOUGH THIS STATEMENT IS REPEATED LATER, I ADDED IT TO MAKE EASIER TO COMPARE WITH THE SEMIONLINE RESULTS IN NEXT PARA.

In \cite{colorful2} it was proved (independently to and after us, but with very similar methods) that for any $c$ and $k$, and any semi-online $c$-coloring algorithm, there exists a finite set of points for which the algorithm produces a monochromatic wedge of size at least $k$, thus, this stronger version of Theorem \ref{thm:noonline} remains true.
In \cite{colorful2} it was also shown that there is no semi-online algorithm that $k$-properly $c$-colors intervals with respect to points.
%On the other hand, the semi-online versions of the results of the next paragraph are open problems.  %FIXME this sentence
% needs to be rewritten

Knowing that a constant number of colors are not enough in the online setting, we can ask for the dependence of the needed number of colors $c$ on the number of points $N$ and on $k$. We consider the cases when either $c$ or $k$ is fixed. For $c=3$ fixed, Theorem \ref{3color} determines exactly the maximum number of points that can always be $k$-properly $3$-colored online, the answer is quadratic in $k$. If $c\ge 4$ is fixed or if $k$ is fixed, the behavior is different, Theorem \ref{constantc} shows that the maximum number of points that can always be $k$-properly $c$-colored is exponential in $ck$. Theorem \ref{constantk} gives an online coloring algorithm which achieves this even without knowing the number of points in advance, that is for an arbitrary point set at an arbitrary step $N$, the $N$ points are $k$-properly $c$-colored using $c=\Theta(\log N/k)$ colors. Recall that the primal and dual problems are equivalent for wedges.

In Section \ref{online-intervals} we show how our results on properly coloring wedges online yield the same results about proper coloring intervals online. Recall that the dual problem of online coloring points with respect to intervals is not equivalent with the primal problem of coloring intervals. Moreover, the online version of the dual problem is not really interesting as it is easy to see that two colors are not enough to properly color points with respect to intervals, whatever we choose $k$, while $3$ colors are already enough for any point set, even for $k=2$ \cite{cd}. Note that this is equivalent to the aforementioned problem of properly coloring points on the diagonal line with respect to wedges.

 %In the former we derive upper bounds on $k$ (in terms of $n$) for which there is always a $k$-proper online $c$-coloring of a set of $n$ points.  For fixed $k$, we derive bounds on $c$ for which there is always a $k$-proper online $c$-coloring of $n$ points.

So far we investigated primal and dual versions of {\em online} coloring wedges and intervals.
In \cite{cd} {\em quasi-online} coloring wedges was investigated (in which case the problems in the primal and dual settings are equivalent). %The last remaining case from Problem \ref{prob1} and Problem \ref{prob2} is quasi-online coloring intervals and its dual, that is, coloring points with respect to intervals.
Similarly as in the case of octants and wedges, coloring quasi-online intervals is equivalent to (offline) coloring bottomless rectangles and also coloring points quasi-online with respect to intervals is equivalent to (offline) coloring points with respect to bottomless rectangles. Both of these were regarded in \cite{keszegh,wcf} and exact results were proved. However, in the primal version the coloring algorithms were overly complicated and computationally not efficient, thus in Section \ref{quasi-online} we revisit this topic and give simpler and efficient algorithms to properly color intervals quasi-online (and thus also to offline properly color bottomless rectangles). The proofs are different from the ones in \cite{keszegh,wcf} and utilize and generalize the tool from \cite{cd} to online build a tree whose offline coloring gives the desired quasi-online coloring. Thus, they also serve as further demonstrations for the usefulness of this tool for quasi-online coloring problems.

%Throughout the paper, when dealing with wedge-coloring, we will rather consider the dual problem of coloring points with respect to wedges, as it is equivalent to the primal and is more convenient. The equivalence of the primal and dual problems for centrally symmetric shapes (and also for wedges) is a well known fact. Briefly, replacing each wedge in the primal setting with its apex and then doing a centrally symmetric reflection of the whole plane, we get a set of points $S$ whose dual hypergraph is exactly the same as the primal hypergraph in the original setting. In particular, $k$-proper $c$-colorings of the dual point set with respect to wedges are in bijection with the $k$-proper $c$-colorings of the original collection of wedges, for more details see e.g., \cite{cdsurvey}. However, for intervals the dual problem is not equivalent to the primal problem.

\section{Online coloring wedges and intervals}\label{online}

\subsection{Online coloring wedges}
Our first result is a negative answer to the question of Tardos \cite{tardos}.

\begin{theorem}\label{thm:noonline}
For every $c$ and $k$ and every online $c$-coloring algorithm there exists an ordered set of
$N=2^{ck}-1$ points for which the algorithm produces a monochromatic wedge of size at least $k+1$.
\end{theorem}

%In other words, no online coloring algorithm using $c$ colors can avoid to make a monochromatic wedge of size $k+1$ for some sequence of $N=2^{ck}-1$ points. 

We start with some definitions.

%To produce these constructions, it will be convenient to be able to refer to the relative positions of sets of points.

\begin{definition}\label{def:south-east}
Let $A$ and $B$ be disjoint sets of points in the plane.  We say $A$ is \emph{south-east} of $B$ if
there exist $x_0,y_0\in\mathbb{R}$ such that 
\begin{enumerate}
\item $\forall (x,y)\in A,\;\; x > x_0 \mbox{ and } y < y_0$, 
\item $\forall (x,y)\in B,\;\; x < x_0 \mbox{ and } y > y_0$.
\end{enumerate} 
\end{definition} 

We will use the other 3 directions, north-west, south-west, and north-east in a similar manner.
%,

The next definition is similar to what was used for an unrelated problem in \cite{panos} and recently (independently) in \cite{colorful2}.

\begin{definition}
For a collection of $c$-colored points in the plane, we define the associated {\em color-vector} to be a vector of length $c$ where the $i^\text{th}$ coordinate is the size of a largest (containing most points) wedge consisting only of points with color $i$.
The {\em size} of the color-vector is the sum of its coordinates.
\end{definition}

To prove Theorem \ref{thm:noonline} we prove a stronger statement, which immediately implies Theorem \ref{thm:noonline}.

\begin{lemma}\label{wedgevector}
For $n\geq 2$ and for any online $c$-coloring algorithm there exists an ordered set of $N=2^n-1$
points for which the algorithm will produce a coloring whose associated color-vector has size at least $n+1$.
\end{lemma}

Given an online coloring algorithm, we show how to explicitly produce such an ordered set of points.
In particular, we give an inductive method of generating the ordered set of points: the position of the
$n^\text{th}$ point will be determined by the coloring the algorithm gives to the first $n-1$ points.  

\begin{proof}%[Proof of lemma \ref{wedgevector}]
We prove by induction on $n$. %the size of the color-vector.
When $n=2$, we must produce an ordered set of $3$ points.
These will all be placed on the line $\ell=\{(x,y)\mid y=-x\}$.
Place the first two points at distinct positions on the line $\ell$.
If they are given the same color by the algorithm, place the third point south-east of the first two (and on the line $\ell$).  Otherwise, if the
first two points are given different colors by the algorithm, place the third point on the line $\ell$ between the first two points.  In either case, the color-vector of the resulting colored point set will have size $3$.

By the inductive hypothesis, using at most $2^{n-1}-1$ points, we can produce a set $S_1$ for which the
algorithm produces a color-vector of size at least $n$. Continuing we can
produce a second disjoint set, $S_2$, south-east from $S_1$ again using at most $2^{n-1}-1$ points for
which the algorithm produces a color-vector of size at least $n$.   If the two color-vectors are
different, then the whole point set $S_1\cup S_2$ has a color-vector of size at least $n+1$.
Otherwise, if the color vectors are the same, then we put an additional point $p$ as follows.  As
$S_2$  is south-east from $S_1$, let $x_0$ and $y_0$ be as in Definition \ref{def:south-east}.  Then
we let $p=(x,y_0)$ where $x=\min\{x\mid (x,y)\in S_1\}$.
Note that $p$ is south-west from $S_1$ and that $S_2$ is south-east from $p$. 
Then as this point is colored with some color, $i$, the $i^{\text{th}}$ coordinate of the
color-vector of $S_1\cup\{p\}$ is one bigger than the $i^\text{th}$ coordinate of the color-vector of $S_1$
(the rest of its coordinates is $0$.) By the monochromatic wedge corresponding to this coordinate
(containing $p$) together with the monochromatic wedges guaranteed by the color-vector of $S_2$, we
get that $S_1\cup S_2\cup\{p\}$ has a color-vector of size at least $n+1$. Altogether we used at
most $2(2^{n-1}-1)+1=2^{n}-1$ points, as desired. 
\end{proof}
%If we want to have such a color-vector after exactly $2^{n}-1$ points, one can always add any number of points far to the south-west from all the other points, their coloring cannot decrease the color-vector.
%The above sentance seems unnecessary as we do not have sharp bounds here.

What happens if $c$ or $k$ is fixed? The case when $c=2$ was considered, e.g., in \cite{cd}.
It is not hard to see that using $2k-1$ points, the size of the largest monochromatic wedge can be forced to be at least $k$ and this is the best possible.
The next theorem states that For $c=3$ exactly $k^2$ points are needed to force a monochromatic wedge of size $k$.

\begin{theorem}\label{3color} 
For any $k>0$ and any online $3$-coloring algorithm there exists a set of $k^2$ points for which the
algorithm produces a monochromatic wedge of size $k$.  This is best possible as for any $k>0$ there exists an online $3$-coloring algorithm which colors any ordered set of
$k^2-1$ points without producing a monochromatic wedge of size $k$.
\end{theorem}

Before the proof, we need a few more definitions.

\begin{figure}
\begin{center}
\includegraphics[width=6cm]{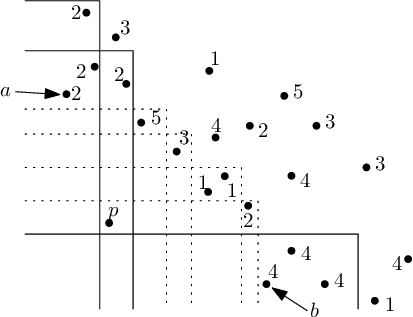}
\end{center}
\caption{The point $p$ borders the $3$ mamos with color $2$ and $4$ denoted by solid lines %(all mamos that $p$ borders contain $a$ or $b$) dom: ez sztem nem igaz
and threatens the mamos denoted by dotted lines. These are the mamos that $p$ is a potential member of.
The points $a$ and $b$ are as required in Proposition \ref{mamoneighbor}.
In a general step of the coloring in the proof of Theorem \ref{constantc} the new vertex $p$ cannot be colored $2$ or $4$. By the
order of our preference $p$ is colored with $3$, thus introducing a (monochromatic) wedge of size $2$.}
\label{fig:addvertex}
\end{figure}

\begin{definition}
Let $X$ be a set of colored points in the plane. %dom: itt nem kell partialt is megengedni?
%Furthermore suppose there is a wedge, $W$, consisting of the points $X=\{x_{i_1}, x_{i_2},\ldots, x_{i_m}\}$.
A non-empty wedge $W$ is {\em maximal monochromatic}, or simply {\em mamo}, if it is monochromatic and there is no monochromatic wedge that contains it.
Two mamos are called {\em neighbors} if they are contained within a larger (non-monochromatic) wedge which contains no other mamo. 
For a new point $p\notin X$ and a (not necessarily maximal) monochromatic wedge $W$ %containing the points $X=\{x_{i_1}, x_{i_2},\ldots, x_{i_m}\}$,
we say that 

\begin{itemize}
\item $p$ {\em threatens} $W$ if all points of $W\cap X$ are north-east from $p$;
\item $p$ {\em borders} $W$ if $p$ does not threaten $W$, and there is a wedge that contains $p$ and some (possibly all) points of $W\cap X$, but no point from $X\setminus W$;
%\item $p$ {\em splits} $W$ if there is a wedge that contains exactly $X\cup\{p\}$, but there is no wedge that contains exactly $W\cap X$;
\item $p$ is a {\em potential member} of $W$ if $p$ threatens or borders $W$.
%\item $p$ is a {\em cut} from $W$ if none of the above holds.
\end{itemize}
\end{definition}

For an example see Figure \ref{fig:addvertex}, where for visual readability mamos are slightly shrinked (by doing that the hyperedges induced by the monochromatic wedges remain the same).

\begin{definition}
If during an online coloring algorithm at time $t$ the point $p$ arrives, then it is initially {\em not destroyed}. Further, after coloring $p$ at step $t$, $p$ {\em destroys} all points (and thus such a point becomes {\em destroyed} at this step) that are north-east from $p$, were not destroyed at an earlier step and are colored a different color than $p$. If $p$ gets a color that differs from the color of a point south-west from it, then $p$ is also destroyed (by this point).

Similarly, a wedge $W$ which is monochromatic before step $t$ is {\em destroyed} by $p$ in step $t$ if after step $t$ there is no monochromatic wedge with the same set of points as $W$. %dom 05.09: I rewrote this def as it was bad before and I think that we only use it in this sense.
%A new (uncolored) point $p$ is a {\em potential member} of $W$, if there exists a wedge consisting of the points
%$X\cup\{p\}$.
%Furthermore $p$ {\em threatens} $W$ if every wedge containing $X$ also contains $p$.
%If $p$ threatens a monochromatic wedge $W$ and is given a color distinct from the color of $W$, then
%we say that the wedge $W$ and its points are {\em destroyed}.
%If $p$ is a potential member of $W$ but does not threaten $W$, then $p$ is {\em at the border} of $W$. 
\end{definition}

Note that during an online coloring algorithm a point is destroyed at most once and when it is destroyed it cannot be in a monochromatic wedge anymore.
%ITT LEHET UGYE ILYEN IS
%P 
%  k
%P   P

%\begin{proposition}\label{mamopath} If for a partially colored point set we represent mamos by a graph where we connect neighbors by an edge, we get a collection of vertex disjoint paths.
%\end{proposition}
%\begin{proof} A point $p$ is contained in some mamo if and only if the wedge whose apex is $p$ is monochromatic.
%Two such points with a different color must be positioned north-west-south-east from each other.
%KESZEGHnek: IDE IS KENE EGY ABRA
%Grouping the same colored points into groups that are contained in a monochromatic wedge, we get all mamos.
%\end{proof}

\begin{proposition}\label{mamoneighbor} In a colored point set $X$, if for a point $p\notin X$ there is no point of $X$ south-west from $p$, then there are two colors, such that any mamo that $p$ borders, are of one of these colors. %thus $p$ threatens every other mamo that contains it. 
\end{proposition}
\begin{proof}
Let $a$ be a southern-most point of the set of points in $X$ which are north-west of $p$.
Let $b$ be a western-most point of the set of points in $X$ which are south-east of $p$.  
Given a mamo $W$ of $X$ that contains $p$, if it contains $a$ or $b$, then the color of $W$ is the same as the color of $a$ or $b$.
Otherwise, $W$ contains only points of $X$ that are north-east from $p$, i.e., $W$ is is threatened by $p$ (and thus not bordered by $p$). For an example see Figure \ref{fig:addvertex}.
\end{proof}

The colors of the mamos bordering a point are denoted the \emph{border colors} in the sequel.

\begin{proof}[Proof of Theorem \ref{3color}]
To prove the first statement, %first consider the situation when there are 
we will maintain three (possibly empty) monochromatic wedges, $L_t$, $M_t$, and $R_t$, of distinct colors, such that %without loss of generality, dom: indeed we could, but I think it's best to avoid as it's non-trivial that we can suppose this, reader might be confused.
the points from $M_t$ are south-east of the points from $L_t$ and the points of $R_t$ are south-east of the points of $M_t$.
At the beginning all three wedges are empty, $L_0=M_0=R_0=\emptyset$.
Denote the size of $L_t$, $M_t$, and $R_t$ by $l_t, m_t$, and $r_t$, respectively.
We maintain that ${l_t+1\choose 2} + m_t + {r_t+1\choose 2}$ increases by at least one with the addition of each new point.
This implies that after $k^2$ steps at least one of the values will be $k$, since if $l_t=m_t=r_t\le k-1$, then the expression is at most ${k\choose 2} + (k-1) + {k\choose 2}=k^2-1$.

If at any time $m_{t}>r_{t}$, then we change to the wedges $L_t=L_{t}$, $M_t=\emptyset$, and $R_t=M_{t}$.
This preserves the condition and ${l_t+1\choose 2} + m_t + {r_t+1\choose 2}$ cannot decrease. 
We similarly proceed if $m_t>l_t$. Thus we can assume that $m_{t}\le\min \{l_{t},r_{t}\}$.

\begin{figure}
\begin{center}
\includegraphics[height=6cm]{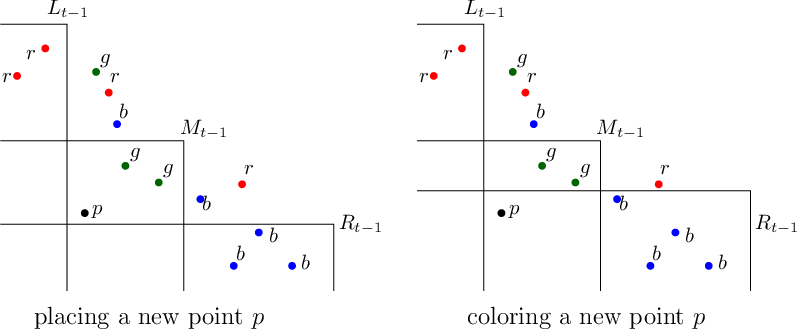}
\end{center}
\caption{First and second part of the Proof of Theorem \ref{3color}}\label{fig:add3col}
\end{figure}

 We place a new point $p$ south-west of the points in $M_{t-1}$ but south-east from the points of $L_{t-1}$ and north-west from the points of $R_{t-1}$ (see left of Figure \ref{fig:add3col}).
This way $p$ is a potential member of all three wedges but only threatens $M_{t-1}$.
For any coloring of $p$, we have to pick $L_t$, $M_t$, and $R_t$ such that ${l_t+1\choose 2} + m_t + {r_t+1\choose 2}>{l_{t-1}+1\choose 2} + m_{t-1} + {r_{t-1}+1\choose 2}$.

If $p$ is given the color of $M_{t-1}$, then let $L_t=L_{t-1}$, $M_t=M_{t-1}\cup \{p\}$, and $R_t=R_{t-1}$, the sum increases by one.
Otherwise $p$ is colored with the color of, say, $L_{t-1}$.
In this case let $L_t=L_{t-1}\cup \{p\}$, $M_t=\emptyset$, and $R_t=R_{t-1}$.
Now 
\begin{align*}
{l_t+1\choose 2} + m_t + {r_t+1\choose 2}&= {l_{t-1}+2\choose 2} + 0 + {r_t+1\choose 2}\\
&={l_{t-1}+1\choose 2} + l_{t-1} +1 + {r_{t-1}+1\choose 2}\\
&\ge {l_{t-1}+1\choose 2} + m_{t-1}+1 + {r_{t-1}+1\choose 2}.
\end{align*}

To prove the second statement, we must assign colors online to at most $k^2-1$ points to avoid a monochromatic wedge of size $k$.
If at time $t$ the new point $p$ is north-east from some earlier point $q$ then we color it to a different color from $q$, this way we do not introduce new monochromatic wedges (but may destroy some). Otherwise, using Proposition \ref{mamoneighbor}, when the new point, $p$, arrives, there are two colors that bordering mamos can have.
Consider the largest size of these of each color and denote them by $L_{t-1}$ and $R_{t-1}$, and their sizes by $l_{t-1}$ and $r_{t-1}$ such that $l_{t-1}\le r_{t-1}$, and let their colors be red and blue. %NEM KELL IDE WLOG, MERT EN DONTOM EL, HOGY HOGY JELOLOM OKET 
Denote by $m_{t-1}$ the size of the largest mamo $M_{t-1}$ threatened by $p$ having the third color, green. See right of Figure \ref{fig:add3col}.
If $l_{t-1}\le m_{t-1}$, then color $p$ red, otherwise color it green. That is, we always color $p$ to the color of a smallest mamo among the three differently colored mamos $L_{t-1},M_{t-1},R_{t-1}$ of which $p$ is a potential member.

We claim that if $a_{t}\le b_{t}$ are the sizes of the largest pair of mamos with different colors at the end of the step at time $t$, then at least ${a_t\choose 2} +{b_t\choose 2}$ points have been destroyed by the end of the step at time $t$.
To prove this, first see that at one step we add and color only one point thus at most one of $a_t$ and $b_t$ can increase and only by at most $1$. Further, 
${a_t\choose 2} +{b_t\choose 2}$ can be greater than ${a_{t-1}\choose 2} +{b_{t-1}\choose 2}$ only if all three of $m_{t-1}$,
$l_{t-1}$, and $r_{t-1}$ are at least $a_{t-1}$, and at least two of them is equal to $a_{t-1}$. %: by definition $l_{t-1}\le r_{t-1}$, thus for ${a_t\choose 2} +{b_t\choose 2}>{a_{t-1}\choose 2} +{b_{t-1}\choose 2}$ it is necessary that $l_{t-1}=a_{t-1}\le m_{t-1}$.
In this case $\{a_t,b_t\}=\{a_{t-1}+1,b_{t-1}\}$ (as an unordered pair of integers) and we color the new point $p$ red and destroy at least
$a_{t-1}={a_{t-1}+1\choose 2} +{b_{t-1}\choose 2}-\left({a_{t-1}\choose 2} +{b_{t-1}\choose 2}\right)$ green points, proving the claim.

Suppose the first time we obtain a mamo of size $k$ is at the end of step $t$, then we must have $a_{t-1}=b_{t-1}=k-1$, thus we destroyed by the end of step at time $t-1$ at
least $2{k-1\choose 2}$ points. Further, we must have $l_{t-1}=r_{t-1}=m_{t-1}=k-1$, thus the three differently colored mamos $L_{t-1},M_{t-1},R_{t-1}$ contain $3(k-1)$ non-destroyed points. %Additionally, in this step $t$ we destroy $a_{t-1}=k-1$ further points not contributing to $a_t$ nor $b_t$. Together with the $a_{t-1}+b_{t-1}\ge 2(k-1)$ points in the largest pair of mamos with different colors before step $t$ and 
Together with the $1$ point we add at step $t$, in total we have at least $2{k-1\choose 2}+3(k-1)+1=k^2$ points.
\end{proof}

%In particular, using $\Theta(k^2)$ points, one can force a monochromatic wedge of size $k$ under any 3-coloring of the points.  This is best possible.

For $c\ge 4$ we can give an exponential (in $ck$) lower bound for the number of points we can color. %, which gives the same order of magnitude as the upper bound in \ref{wedgeintervalobs}.
%\dom{ez mi akar lenni???}
%To simplify notation we define the function $p_r(k)=\sum_{i=0}^{k}{r^k}=\frac{r^{k+1}-1}{r-1}$.

\begin{theorem}\label{constantc}
For $c\ge 4$ we can color online with $c$ colors any set of at most $1.22074^{c(k-2)+1}$ points such that throughout the process there is no monochromatic wedge of size $k$.
Moreover, if $c$ is large enough, then we can even color online any set of at most $1.46557^{c(k-2)+1}$ points without creating a monochromatic wedge of size $k$.
For large $c$, we can color online any set of at most $1.46557^{c}$ points without creating any monochromatic wedges of size $2$.
\end{theorem}

Before proving the theorem we introduce some notations.
\begin{definition}
Let $\A$ be an online coloring algorithm of points in $\mathbb{R}^2$.  Before step $t$, let 
$p_1,p_2\ldots,p_{t-1}$ be the $t-1$ points the algorithm has colored so far. In step $t$ the point $p_t$ appears and is colored. Define the \emph{weights} of the points after step $t$, $w_t(p_i)$, $i\le t$ as
\begin{equation}
w_t(p_i) =\begin{cases}
& 0, \textrm{ if $p_i$ is destroyed by $p_t$ in step  $t$,}\\
& 1+\sum_{p_j \textrm{ is destroyed by $p_t$ in step } t} w_{t-1}(p_j), \textrm{ if } i=t,\\
& w_{t-1}(p_i) \textrm{ otherwise.}
\end{cases} 
\end{equation}
The \emph{weight of a wedge} is the sum of the weights of the points in the wedge. 

Given an online coloring algorithm $\A$, let $w_\A(i,j)$ be the minimal weight of a monochromatic wedge of size $i$ and color $\ge j$ OR size $>i$ and any color over all online point sets colored using algorithm $\A$. If $\A$ is such that there is never such a wedge, then define $w_\A(i,j)=\infty$.
%Given an online coloring algorithm $\A$, let $A(i,j)$ be the minimal number of points in an online point set which contains a monochromatic wedge of size $i$ and color $j$ when colored by $\A$. If $\A$ is such that there is never such a wedge, then define $A(i,j)=\infty$.
%Given the coloring algorithm $\A$, define $A_{i,j}$ to be the minimal number of points necessary to force a monochromatic wedge of size $i$ and color $j$.  If it is impossible to force such a wedge under $\A$ define $A_{i,j}=\infty$.
\end{definition} 

Note that after any step $t$ a point that was destroyed at any earlier time has zero weight and that the sum of all the weights after any time $t$ is always equal to the number of arrived points $t$. Note that $w_\A(i,1)$ is a lower bound on the minimal number of points in an online point set which contains a monochromatic wedge of size $i$ when colored by $\A$. %Note that $A(i,j)$ is the minimal number of points that forces $\A$ to introduce a monochromatic wedge of size $i$ and color $j$.

\begin{proof}[Proof of Theorem \ref{constantc}]
First we describe the algorithm how to color a new point.
If the new point $p$ is north-east from an earlier point, it is given a different color from one such point.
In this case no new monochromatic wedge, and in particular, no new mamo is created.  
%Otherwise the new point will eventually be part of a new mamo, $W$.
%We ensure that the color of $W$ is distinct from its neighbors' colors.
%In particular, let $T$ be the set of monochromatic wedges which are threatened by the new point.
%Then there are at most two wedges, $N_1$ and $N_2$, which are neighbors of some wedge in $T$ (they are not necessarily both neighbors of the same wedge) and are not threatened by the new point.
Using Proposition \ref{mamoneighbor}, for any new point $p$, there are at most two colors (the border
colors) such that any mamos that $p$ borders, are of one of these colors.
Every other mamo containing $p$ is threatened.
%From the $c$ colors we disregard these two colors of the wedges $N_1$ and $N_2$.
Choose from the non-border colors the color which first minimizes the size of the largest mamo containing $p$ and secondly minimizes the color (as a number from 1 to $c$).
Thus our preference is first to have a size $1$ wedge of color $1$, then a size $1$ wedge of color $2$,
$\ldots$, size $1$ wedge of color $c$, size $2$ wedge of color $1$, $\ldots$ etc. (with border colors excluded).
Figure \ref{fig:addvertex} is an illustration of the algorithm for a 5-coloring.

We now show that this coloring algorithm $\A$ can indeed color the requisite number of points without creating a wedge of size $k$ or more.
% We inductively define a {\em weight} for points as follows.
%dom: itt ha wedge-nek van weightje, akkor baj van az osszelogo, azonos szinu wedge-ekkel.
% The weight of a point $p$ is one plus the number of points that are destroyed when we color $p$,
% i.e., the number of points north-east to $p$ that have a different color at the arrival of $p$ that
% have not been previously destroyed.
% Thus at every step the sum of the weights of the non-destroyed points will be equal to the total number of points placed.
% Note that points which are created without destroying other points have weight $1$.
% The weight of a wedge is simply the sum of the weights of the points in it.
% Given our coloring algorithm, let $A_{i,j}$ denote the smallest weight achievable in the worst case for a wedge of size $i$ and color $j$.
% In some sense, this measures how many points are required to ``build'' such a wedge.
We will give a lower bound on $w_\A(i,j)$ as $w_\A(k,1)$ lower bounds the number of points required to make a monochromatic wedge of size $k$.
To simplify notation, define $b_{c(i-1)+j}=w_\A(i,j)$ for $1\leq j\leq c$ and $b_n=0$ for $n\le 0$.
%Denoting this algorithm as $\A$, recall that $A(i,j)$ is the minimal number of points that can force $\A$ to introduce a monochromatic wedge of size $i$ and color $j$. To simplify notation, define $b_{c(i-1)+j}=A(i,j)$ for $1\leq j\leq c$ and $b_i=0$ for $i\le 0$.
%As $1\leq j\leq c$ there is a bijection $a_{i,j}\leftrightarrow b_{c(i-1)+j}$. 
It follows from the definition that $b_i\le b_j$ if $i\le j$.
Our goal is to give a good lower bound on $w_\A(k,1)=b_{c(k-1)+1}$.

For a mamo of size $i$ and color $j$, order its points in the order as they appeared $p_1,\dots p_i$. Consider the time when the $h^{\text{th}}$ point, $p_h$ is added as a new point.
Notice that when $p_h$ arrives as a new point and we color it, all points of many, previously monochromatic, possibly intersecting wedges will be destroyed.
More precisely, from our preferences we have that all points of at least $c-3$ mamos of different colors that are ``almost as big'' as the one we create by adding $p$ are destroyed.
These sizes are at least $h-1$ (and can be more, as when we add $p$ we might create a bigger mamo than $h$) and can be best expressed with the below formula using $b_n$.
%Also, we merge into the new mamo the points of a mamo of the same color whose size is at most one smaller (possibly even larger, as later some of the points might be destroyed).
Denote $r=c(h-1)+j$.
From the above, using that $b_n$ is a monotone increasing sequence, we can show that after adding $p_h$, its weight is at least $1+b_{r-3}+b_{r-4}+\ldots+b_{r-c+1}$.
Note that $b_{r-1}$ and $b_{r-2}$ are missing from this sum; this is because we had to choose the best available color that is not one of the (at most) two colors of a bordering mamo.
This leaves $c-2$ options, of which the mamo whose weight is smallest has weight at least $b_{r-c}$, the next has weight at least $b_{r-c+1}$ and so on, the last has weight at least $b_{r-3}$.
Therefore, after choosing the best color, we get that at this step $p_h$ destroys at least $b_{r-3}+b_{r-4}+\ldots+b_{r-c+1}$ points, that is 
$$w(p_h)\ge b_{r-3}+b_{r-4}+\ldots+b_{r-c+1}.$$
% because the coloring algorithm must choose a color different from the colors of the mamos that the new point $p$ borders, that is the colors of the neighboring mamos of the eventual mamo containing $p$.
%It is also important to remark that we might simultaneously create several (intersecting) mamos (of the same color) by coloring one point.
%This, however, causes no problems as we always count the weight of the destroyed points, not the (maximal) monochromatic wedges.
%We might also truncate some bordering mamos (that are not of the color of $p$) thereby creating new mamos, AND WHY DON'T THESE CREATE A PROBLEM?

Therefore, by lower bounding the sum of the weights of the points in a mamo of size $i$ and color $j$, we lower bound $b_{c(i-1)+j}$ and get

$$b_{c(i-1)+j}\ge \sum_{h=1,\dots ,i} w(p_h)\ge \sum_{r=j,c+j,\ldots,c(i-1)+j} 1+b_{r-3}+b_{r-4}+\ldots+b_{r-c+1}.$$

%By induction, this gives $b_r \ge 1+b_{r-3}+b_{r-4}+\ldots+b_{r-c}$. megsem, mert rossz fele all

To get a lower bound for $b_r$, we introduce $a_r$ with a non-homogeneous linear recursion $a_r = 1+a_{r-3}+a_{r-4}+\ldots+a_{r-c}$ starting with $a_r=0$ for $r\le 0$. Notice that for $r\le c$ we  even have $a_r = 1+a_{r-3}+a_{r-4}+\ldots+a_{r-c+1}$ as $a_{r-c}=0$ in this case.
We claim that we have $a_r\le b_r$: for $r\le 0$ this follows from the definition and for $r>0$ we can use induction with the above formula to get 

$$b_{c(i-1)+j}\ge \sum_{r=j,c+j,\ldots,c(i-1)+j} 1+b_{r-3}+b_{r-4}+\ldots+b_{r-c+1} \ge$$
$$\sum_{r=j,c+j,\ldots,c(i-1)+j} 1+a_{r-3}+a_{r-4}+\ldots+a_{r-c+1}=
\sum_{r=j,c+j,\ldots,c(i-1)+j} a_{r}-a_{r-c}= a_{c(i-1)+j}.$$

We can reduce the above recursion for $a_r$ to a homogeneous linear recursion in several standard ways: by using $a_{r+1}-a_r=a_{r-2}-a_{r-c}$; or if $c\ge 4$, by defining $a_r'=a_r+1/(c-3)$; or simply omitting the additive term as it anyway does not affect the order of magnitude.
Using either of the above, we can conclude that the magnitude of $a_r$, and thus of $b_r$, is at
least $q^r$ where $q$ is the real, root of the equation
\begin{equation}\label{eq:root-q}
q^{c}-q^{c-3}-q^{c-4}\ldots-1=0,
\end{equation} 
 which is unique if $c\ge 4$.  In order to get an explicit lower bound, notice that $a_r\ge 1\ge
q^{r-c}$ if $1\le r\le c$, and so from the recursion using (\ref{eq:root-q}) we also have $a_r\ge q^{r-c}$ for all $r$.
Among $c\ge 4$, this root is the smallest for $c=4$, in which case we look for the root of $q^4-q-1=0$, and a simple numerical calculation shows that this number is larger than $1.22074$.
Therefore we have $b_{c(k-1)+1}\ge 1.22074^{c(k-1)+1-c}$, just what was needed.
In fact, for $c=4$, the sequence we get is $a_r = 1+a_{r-3}+a_{r-4}$. %\cite{OEIS2},
%a shifted version of the sequence $T_r = T_{r-3}+T_{r-4}$, which can be interpreted as the third member of the series of sequences starting with the well-known Fibonacci sequence $F_r = F_{r-1}+F_{r-2}$, followed by the Padovan sequence $P_r = P_{r-2}+P_{r-3}$. %, which is http://oeis.org/A079398 shifted.
Using standard methods, from the recursion we could determine the exact asymptotics of $a_r$ for any $c$.
%The exact order of magnitude can be determined similarly for other $c$.
As $c$ tends to infinity, the sequence $b_r$ is getting closer and closer to the Narayana's cows sequence $N_r = N_{r-1}+N_{r-3}$ \cite{OEIS}, and $q$ tends (from below) to the real root of $q^3-q^2-1=0$, which is bigger than $1.46557$.
Therefore, $a_r\ge 1.46557^r$, if $c$ is large enough.
From this we obtain that $a_{c(k-1)+1}\ge 1.46557^{c(k-2)+1}$ if $c$ is large.

To obtain bounds for $k=2$, we should be more careful with the first few terms of the sequence as our initial estimate $a_r\ge 1\ge q^{r-c}$ for $1\le r\le c$ is not strong enough.
The values for small $c$ can be calculated manually, while for larger values we can use the exact value for $N_r$, the $r^{\text{th}}$ term of the Narayana's cows sequence (by Benoit Cloitre \cite{OEIS}) to conclude that $a_r>0.6\cdot 1.465571^{r-1}-0.5$.
From this we obtain that $b_{c+1}\ge 1.46557^{c}$ if $c$ is large.
\end{proof}

In conclusion, if there are $c\ge 4$ colors, the smallest  number of points that could force a monochromatic wedge of size $k$ is exponential in $ck$.
%Reversing the equality we get almost exact results about the needed number of colors or about the required $k$. \dom{ezt nem ertem}
These bounds can be used to estimate the size of the largest monochromatic wedge when coloring $N$ points with $c$ colors to be $\Theta(\log N/c)$ in the worst case.
Similarly for fixed $k\geq 1$, to avoid a monochromatic wedge of size $k$ when coloring $N$ points, $\Theta(\log N/k)$ colors are necessary and sufficient.

\begin{corollary}\label{constkcor}
There is an algorithm to color online $N$ points in the plane using $\Theta(\log N/k)$ colors such that all monochromatic wedges have size strictly less than $k$.
\end{corollary}

Recall that Lemma \ref{wedgevector} stated that $N=2^{n}-1$ points can always force a size $n+1$ color-vector.
Theorem \ref{constantc} implies a lower bound close to this upper bound. Indeed, fix, e.g., $c=4$
and $k=\lceil n/4\rceil$. If the number of points is at most $N=O(1.22074^{n})=O(1.22074^{ck})$,
then by Theorem \ref{constantc} there is an online coloring such that at any time there is no
monochromatic wedge of size $k$, thus the size of the color-vector is always at most $4(k-1)<n$.

%When we defined the coloring process in the proof of Theorem \ref{constantc} we considered the number of points as a number known in advance. Thus if one of $c$ and $k$ was given, we could compute the other and then run the coloring algorithm accordingly. We claim that this is not necessary, for example the same algorithm works if $c$ is fixed and the number of points is not known in advance (or alternatively, one can regard this case as a situation when we want a coloring of infinitely many points which performs well at any time). Indeed, we proved that whenever the algorithm has to introduce a monochromatic wedge of size $k+1$ we already have at least $N= \Omega(1.22074^{ck})$ points.

Observe that the coloring algorithm in the proof of Theorem \ref{constantc} was oblivious to $k$, thus in fact it implies the following stronger statement.

%KESZEGH: UGY DONTOTTEM TAN ERDEMES EZT A KOV ALLITAST KULON HANGSULYOZNI. IDE IS BE KELL IRNI VEGSO MAGNITUDEOT.
\begin{corollary}\label{corconstantc}
For fixed $c\ge 4$ we can color a countable set of points such that for any $k$, and any $n<1.22074^{c(k-2)+1}$, the first $n$ points of the set are $k$-properly $c$-colored.
\end{corollary}

This gives an algorithm when $c$ is fixed. Suppose now that $k$ is fixed and we want to use as few
colors as possible without knowing in advance how many points will come, i.e., for fixed $k$ we want to minimize $c$ without knowing $N$. To solve this, we alter our previous algorithm.
(Note that similarly it is possible to adjust the algorithm for the cases when for an unknown $N$ we want to minimize $\min(c,k)$, or $ck$, and the answer is still logarithmic in $N$.)
%FIXME awkward
All this comes with the price of loosing a bit on the base of the exponent.
The following theorem implies that for $k=2$ (and thus also for any $k\ge 2$) we can color online any set of $N=O(1.0905^{ck})$ points and if $k$ is big enough, then we can color any set of $N=O(1.1892^{ck})$ points without a monochromatic wedge of size $k$.

\begin{theorem}\label{constantk}
For fixed $k\ge 1$ we can color a countable set of points such that for any $c$, and any $n<2^{(k-1)\lrceil{(c-3)}{4}+1}-1$, the first $n$ points of the set are $k$-properly $c$-colored.
\end{theorem}

\begin{proof}
We need to define a coloring algorithm and prove that it uses many colors only if there are many points.
Both the coloring and the proof are similar to those in Theorem \ref{constantc}, we only need to change our preferences when coloring.
Because of this the analysis of the performance of the algorithm also differs slightly.
We fix a $c$ and an $N< 2^{(k-1)\lrceil{(c-3)}{4}+1}-1$ for which we will prove the claim of the
theorem (the coloring we define will not depend on $c$ or $N$, but only on $k$). Denote the colors by
the numbers $\{1,2,\ldots,c,\ldots \}$. %We group them into consecutive $4$-tuples. %The number of groups is denoted by $m=\lfloor l/4\rfloor$.

We again avoid the colors of the mamos that border the new point $p$.
Denote by $c_p$ the color to be assigned to $p$ (as an integer.)
Our primary preference now is that we want to keep $\lfloor (c_p-1)/4\rfloor$ small. That is, we use one of the four colors $c_p$ that minimizes $\lfloor(c_p-1)/4\rfloor$ under the condition that using one of these colors we can avoid a monochromatic wedge of size $k$. 
%This preference partitions the colors into $4$-sets such that colors within the same set will behave in the same way regarding this algorithm.
Once we have these four colors, our secondary preference is that we choose the color from these four colors that minimizes the size of the largest mamo containing $p$.
%the size of the biggest mamo containing $p$ should be small.
This means that our order of preference is first to have size $1$ wedge of color $1$, $2$, $3$, or $4$, then a size $2$ wedge of color $1$--$4$, $\ldots$, a size $k-1$ wedge of color $1$--$4$, then a size $1$ wedge of color $5$--$8$, $\ldots$ etc.
These rules determine our algorithm, except for the choice when more than one of the four colors is possible according to our secondary preference, in which case the chosen color can be arbitrary.

Again let $\A$ be this algorithm and recall that $w_\A(i,j)$ refers to the minimal weight of a monochromatic wedge of size $i$ and color $\ge j$ or size $>i$ and any color in an online point set colored using algorithm $\A$.
To simplify notations let $b(i,j)=w_\A(i,4j)$.
Note that by the definition of $\A$, it makes no difference whether we consider $4j-3$, $4j-2$, $4j-1$, or $4j$, we get the same values, that is $w_\A(i,j)=b(i,\lrceil{j}{4})$.
Because of this, when it makes no difference, we simply write $4j$ for the color.

We only need to prove that $b(i,j)\ge 2^{(k-1)(j-1)+i}-1$ as this means that if the algorithm uses the color $c+1$, then we had at least $w_\A(1,c+1)=b(1,\lrceil{(c+1)}{4})\ge 2^{(k-1)\lrceil{(c-3)}{4}+1}-1>N$ points, a contradiction. %KB: ES EZ A KEPLET JO? A TETELBEN MEG N DEFINICIOJABAN JAVITOTTAM! 
We prove by induction, $b(1,1)=1$.
If we introduce a wedge of size $i>1$, we have to destroy all points of at least one mamo of size
$i-1$ that had a different color from the same $4$-set, and merged an old mamo of size $i-1$ that had the same color to the new mamo.

Similarly to the proof of Theorem \ref{constantc}, for a mamo of size $i$ and color $4j$, order its points in the order as they appeared $p_1,\dots p_i$ and consider the time when the $h^{\text{th}}$ point, $p_h$ is added as a new point.
From the above argument, we get $w(p_h)\ge 1+b(h-1,j)\ge 2^{(k-1)(j-1)+h-1}$ by induction.

Therefore, by lower bounding the sum of the weights of the points in a mamo of size $i$ and color $4j$, we get

$$b_{i,j}\ge \sum_{h=1}^i w(p_h)\ge \sum_{h=1}^i 2^{(k-1)(j-1)+h-1}=2^{(k-1)(j-1)+i}-1.$$

Finally, we have to check what happens if $i=1$, i.e., if we color a point $p$ with a color $4j$ that has not yet been used.
In this case we have to destroy all points of at least two mamos of size $k-1$ having a color from $4j-7$, $4j-6$, $4j-5$, or $4j-4$.
Thus $w(p)\ge 1+2b(k-1,j-1)\ge 1+2\cdot (2^{(k-1)(j-2)+k-1}-1)=2^{(k-1)(j-1)+1}-1$.
This means that any point of a wedge of color $4j$ has at least this much weight, therefore $b(1,j)\ge 2^{(k-1)(j-1)+1}-1$.
\end{proof}

\begin{proposition}\label{timeonline}
The online coloring algorithms guaranteed by the second part of Theorem \ref{3color}, Theorem \ref{constantc} and Theorem \ref{constantk} run in $O(n\log n)$ time to color the first $n$ points (even if we have a countable number of points and $n$ is not known in advance).
\end{proposition}

The proof of this proposition is omitted as it follows easily from the analysis of the algorithms.

\subsection{Online coloring intervals}\label{online-intervals}

This section deals with the following {\em interval coloring problem}.
Given a finite family of intervals on the real line, we want to color them online with $c$ colors such that
throughout the process if a point is covered by at least $k$ intervals, then not all of these intervals have the same color.

\begin{proposition}
The (online, quasi-online, semi-online) interval coloring problem is equivalent to a restricted case of the problem of $k$-properly (online, quasi-online, semi-online) coloring points with respect to wedges, where we consider only wedges whose apex is on the diagonal line (defined by $y=-x$).
\end{proposition}

\begin{proof}
%Consider the natural bijection of the real line and diagonal line that maps $(x,-y)$ to $x$.
Associate to every point $p$ of the diagonal the wedge with apex $p$ and associate with every interval $I=((x_1,-x_1),(x_2,-x_2))$ of the diagonal line the point $(x_1,-x_2)$.
It is easy to see that $p\in I$ if and only if the point associated to $I$ is contained in the wedge associated to $p$.
%There is a bijection $\phi$ from the finite intervals of $\mathbb{R}$ to points in the plane below the line $L$ such that if $v\in\mathbb{R}$ and $I_1, \ldots, I_m$ are the intervals covering $v$, then there exists a wedge in the plane with apex on the line $L$ which consists exactly of the points $\phi(I_1), \ldots, \phi(I_m)$.  Indeed, identify $\mathbb{R}$ and $L$ so that each interval lives on $L$. Then any interval $I$ interval uniquely defines a square in the plane whose diagonal is $I$. Let $\phi(I)$ be the south-west vertex of this square. Now it is easy to check that if a point $p$ on $L$ is covered by the intervals $I_1,I_2,\ldots,I_m$ then the wedge with apex $p$ covers exactly the points $b(I_1),\ldots b(I_m)$, as claimed.
\end{proof}

\begin{corollary}
Any upper bound on the number of colors necessary to (online, quasi-online, semi-online) color wedges in the plane is also an upper bound for the number of colors necessary to (online, quasi-online, semi-online) color intervals in $\mathbb{R}$.
\end{corollary}

Also the lower bounds of Theorem \ref{thm:noonline} and of Theorem \ref{3color} follow for intervals easily by either repeating the proofs for intervals or by the following observation.

\begin{observation}\label{wedgeintervalobs}
The proofs of  Theorem \ref{thm:noonline} and of the first part of Theorem \ref{3color} can be easily modified such that all the relevant wedges have their apex on the diagonal line.
\end{observation}

In particular, we have the following.

\begin{corollary}
There is an algorithm to color online $N$ intervals in $\mathbb{R}$ using $\Theta(\log N/k)$ colors such that for every point $x$, contained in at least $k$ intervals, there exist two intervals of different colors containing $x$.
\end{corollary}

As we have seen, positive results for intervals follow directly from the corresponding results for wedges.
Thus all the statements we proved hold for online coloring wedges, also hold for intervals, however, it seems unlikely that the exact bounds are the same.
Thus, we would be happy to see (small) examples where there is a distinction.
As the next section shows, there is a difference between the exact bounds for quasi-online coloring wedges and intervals. 

\section{Quasi-online coloring intervals}\label{quasi-online}

In this section we consider proper quasi-online coloring an ordered collection of intervals $\{I_t\}_{t=1}^{n}$, i.e., proper quasi-online coloring of the dual hypergraph defined by these intervals.

\begin{theorem}\label{2col} Any finite ordered family of intervals on the line can be quasi-online $3$-properly $2$-colored, i.e., colored quasi-online with red and blue such that at any time, for any point contained in at least $3$ intervals, at least one of these intervals is red and another one is blue.
\end{theorem}

We exploit an idea used in \cite{cd}; instead of coloring online the intervals, we online build a labelled acyclic graph (i.e., a forest) with the following properties. At any time $t$, each vertex of the current graph corresponds to an interval on the line, such that for every {\em original} interval (i.e., an interval in $\{I_i\}_{i=1}^{t}$) there is a corresponding vertex. There might be other vertices in the graph corresponding to {\em auxiliary} intervals.
In notation, we usually identify a vertex with the corresponding interval, without causing confusion. The final coloring of the intervals is then generated from this graph. In particular, to define a $2$-coloring, we assign each edge in the forest one of two labels, ``different'' or ``same.''  For an arbitrary coloring of exactly one vertex in each component (tree) of the graph, there is a unique extension to a coloring of the whole graph {\em compatible} with the labelling, i.e., such that each edge labelled ``same'' is incident to vertices of the same color and each edge labelled ``different'' is incident to vertices of different colors. We say that a property is {\em forced} by the labelling if every compatible coloring has this property. At the end we prove that the forest we built forces the original intervals to be $3$-properly $2$-colored at any time.
In \cite{cd} all the edges were labelled ``different,'' so it was actually a simpler variant of our current scheme. As we will see, this idea can also be generalized to more than two colors.

We denote the color of an interval $I$ by $\phi(I)$, the left endvertex of $I$ by $l(I)$ and the right endvertex of $I$ by $r(I)$. These vertices are real numbers, and so they can be compared. We can suppose that they are all different, as slightly perturbing them can only make the coloring problem more difficult.

\begin{proof}[Proof of Theorem \ref{2col}]
Let $\{I_t\}_{t=1}^{n}$ be the given ordered family which has to be quasi-online $3$-properly $2$-colored. We first build the forest and then show that any coloring compatible with this forest is a quasi-online $3$-proper $2$-coloring of the original intervals, as required. As we build the forest we also maintain a family of intervals (corresponding to a subset of the vertices of the forest), called the {\em active} intervals.
The family of active intervals will change during the process.
An interval corresponding to a vertex is not necessarily one of the original intervals, $\{I_t\}_{t=1}^{n}$, it might be an auxiliary interval created during the process.
At any time $t$ the vertices of the actual forest correspond to the first $t$ original intervals and the auxiliary intervals created up to time $t$.
After coloring the forest of the original and auxiliary intervals, we get the desired coloring. %(not necessarily a subfamily of the original intervals).
We maintain that the following properties hold any time, i.e., for any $t$ after adding interval $I_t$ and running step $t$ of the forest-building algorithm (defined later) the following properties hold.

\begin{enumerate}
\item Every point of the line is covered by at most two active intervals.
\item No active interval contains another active interval.
\item For any point on the line (at least) one of the following holds.
\begin{itemize}
\item[(a)] The point is contained in the same number of active intervals as original intervals, and additionally the labelling forces these original
intervals and these active intervals to have the same set of colors.
\item[(b)]
The labelling forces that the point is contained in original intervals of different colors.
\end{itemize}
\item The graph is a forest and each tree in the forest contains exactly one vertex that corresponds to an active
interval.
\end{enumerate}

During the forest-building process for an arbitrary point property 3(a) will hold until some moment and then property 3(b) will hold from that moment on.
Note that if for some point property 3(b) holds, then it will remain so later, as adding vertices and edges to the graph cannot ruin property 3(b).
By property 1, this will guarantee that points in at least $3$ original intervals are contained in both colors.
%Also notice that in case of property 3(b) we always have at most two active intervals covering that point (by property 1), and in this case one or two of these active intervals can actually coincide with one or two of the original intervals covering that point. DOM: Minek ide ez a mondat?
Property 4 ensures that a coloring of the active intervals determines a unique coloring of all the intervals, such that this coloring is compatible with the labelling of the forest. 
Property 2 is a technical condition.

Now we define the forest-building algorithm.
For the first step we simply make the first interval active; our forest will consist of a single vertex corresponding to this interval. In general, at the beginning of step $t$, we have a list of active intervals, ${\cal J}_{t-1}$. Now we add the $t^{\text{th}}$  interval, $I_t$, to the forest. See Figure \ref{figglue} for an example. If $I_t$ is covered by at least one active interval, then we choose one, $J\in {\cal J}_{t-1}$ and connect $I_t$ to $J$ with an edge labelled ``different.'' If there is no active interval containing $I_t$, we add $I_t$ to the family of active intervals. Now if there are active intervals contained in $I_t$, then we deactivate all of them (remove from the family of active intervals) and connect each of them to $I_t$ in the graph with an edge labelled ``different.'' This way properties 2 and 4 remain true. For a point $p$ if property 3(a) did hold before adding $I_t$ then if it is contained in a just deactivated interval, then at this moment of our algorithm property 3(b) will hold for $p$, otherwise property 3(a) remains true (as the possible change in the set of active intervals containing $p$ is the addition of $I_t$, which is an original interval as well).

\begin{figure}
\begin{center}
\includegraphics[width=14cm]{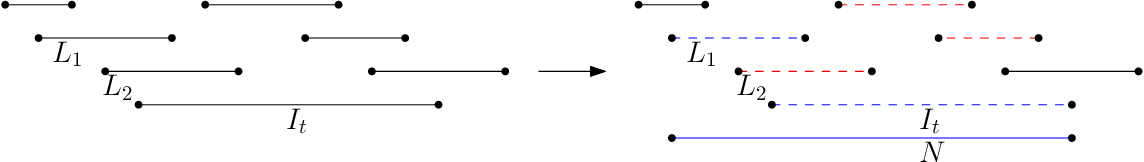}
\end{center}
\caption{A case in the proof of Theorem \ref{2col}. Active intervals are represented by solid lines, deactivated intervals by dashed lines.}
\label{figglue}
\end{figure}

The last thing we do in step $t$ of our forest-building algorithm is to ensure property 1, i.e., that no point is contained in three active intervals.
If there exist such points, they must be contained within $I_t$ (as before adding $I_t$ by the first property there were no such points). 
Using the first property, let $L_1$ and $L_2$ be the (at most) two active intervals covering $l(I_t)$ such that $l(L_1)<l(L_2)$ (if both of them exist).
Similarly, let $R_1$ and $R_2$ be the (at most) two active intervals covering $r(I_t)$ such that $l(R_1)<l(R_2)$ (if both of them exist).
No $L_i$ and $R_j$ can coincide, as such an interval would cover $I_t$. No other active interval can intersect $I_t$, as it would necessarily be completely contained in $I_t$, but all such intervals are already deactivated. Depending on how many of these four intervals exists, we proceed slightly differently.

First, suppose that there is no $L_2$ and $R_2$, only $L_1$ and $R_1$.
$L_1$ and $R_1$ have to intersect, otherwise there is no point covered $3$-fold.
This implies that $L_1\cup R_2$ cover $I_t$.
We deactivate all three of them and add a new interval $N=L_1\cup I_t\cup R_1$ to the graph and make $N$ active.
In the graph, connect $L_1$ and $R_1$ to $N$ with edges labelled ``same''  and $I_t$ to $N$ with edge labelled ``different.''
 %of this case when the active interval $N$ is assigned color blue and deactivated intervals are shown with dashed lines.

Next, suppose that all of $L_1$, $L_2$, $R_1$ and $R_2$ exist.
%Without loss of generality, we can assume that both $L_1$ and $L_2$ exist.
%If $R_1$ and $R_2$ also both exist,
Deactivate $L_1$, $L_2, I_t, R_1$ and $R_2$, and activate a new interval $N=L_1\cup I_t\cup R_2$.
In the graph, connect $L_1, I_t$ and $R_2$ to $N$ with edges labelled ``same.''  Connect $L_2$ and $R_1$ to $N$ with edges labelled ``different.''

Otherwise, without loss of generality, suppose that $L_1$, $L_2$ exists and $R_2$ does not exist ($R_1$ may or may not exist).
Deactivate $L_1$, $L_2$ and $I_t$, and connect them to the new active interval $N=L_1\cup I_t$ again with the edges of $L_1$ and $I_t$ labelled ``same'' and the edge of $L_2$ labelled ``different.''
(Notice that in fact this last case is the same as the first, with $L_2$ playing the role of $I_t$.)

This way we ensured that property 1 holds and it is easy to check that properties 2 and 4 remain true. Similarly as above, it is easy to check that if for a point property 3(a) did hold then now either property 3(a) or 3(b) holds. Recalling that property 3(b) cannot be ruined we get that property 3 remains true for every point.
%This way within a given step, any point which is contained in (at least) two intervals deactivated during this step, is forced by the labelling to be contained in intervals of different colors. For any other point $v$ the number of original intervals containing $v$ remains the same as the number of active intervals covering $v$ (both remains the same or both increases by $1$). The first three properties were maintained and also it is easy to check that the graph remains a forest such that in each component there is a unique active interval.

By property 4, at the end of the process any coloring of the final family of active intervals extends to a coloring of all the intervals (compatible with the labelling of the graph). We have to prove that for this coloring at any time $t$ any point $p$ contained by at least $3$ of the original intervals is non-monochromatic. At time $t$, by property 1 we have that for such a $p$ property 3(a) cannot hold (as we cannot have $3$ active intervals covering $p$). Thus property 3(b) must hold, which is exactly that $p$ is contained in original intervals of both colors.
%By induction at any time $t<n$ the coloring is compatible with the graph at that time, thus by induction any point contained by at least $3$ of these intervals is non-monochromatic. Now at time $n$, if the active intervals are colored, the extension (by induction) is such that every point not in $I_n$ is either covered by at most two original intervals or it is covered by intervals of both colors. On the other hand, from the way we defined the graph, we can see that points covered by $I_n$ and contained in at least $3$ intervals are covered by intervals of both colors as well. Indeed, by the properties maintained, if a point $v$ is not covered by intervals of both colors, then it is covered by as many active intervals as original intervals. Yet, no point is covered by more than $2$ active intervals at any time, thus $v$ is covered by no more than $2$ active and thus no more than $2$ original intervals, so we do not have to worry about it.
\end{proof}

Before proceeding with the next proof we define explicitly the structure that the active intervals have during our algorithms (in the previous and in the next proof as well).
We call an ordered family of intervals $J_1,\dots J_l$ a {\em chain} if $l(J_i)<l(J_j)$ and $r(J_i)<r(J_j)$ for every $i<j$ and $J_i\cap J_{i+1}\ne \emptyset$ for every $i<l$ and no point of the real line is contained in three of the intervals. Observe that in a chain any interval $J_i$ intersects exactly $J_{i-1}$ and $J_{i+1}$ (if they exist). Two chains are {\em disjoint} if the intervals in the first chain is disjoint from the union of the intervals of the second chain.

\begin{theorem}\label{3col} Any finite ordered family of intervals on the line can be quasi-online $2$-properly $3$-colored, i.e., colored quasi-online with $3$ colors such that at any time for any point $p$ contained by at least $2$ of the intervals, the intervals covering $p$ are not all of the same color.
\end{theorem}

\begin{proof}
We again build an edge-labelled graph $G$, in which vertices correspond to (original and auxiliary) intervals and the label of a directed edge is again one of two labels, ``different'' or ``same.'' %For an arbitrary vertex, it either has at most one in-going edge labelled ``same'' or at most two in-going edges labelled ``different''. %A $3$-coloring of the intervals (original and auxiliary) is {\em compatible} with $G$ if the coloring is compatible with $G$ as an undirected graph (as in the definition in the previous proof). 
Again some of the intervals are active.
An order on the non-active intervals is {\em appropriate}, if putting all active intervals arbitrarily ordered at the end of this order we get an order for which every non-active interval is {\em appropriate}: either has at most two forward edges labelled different or at most one forward edge labelled same. 
Suppose that we could color the active intervals compatibly with $G$, then given an appropriate ordering of the non-active intervals, it is easy to $3$-color these intervals in backwards order to get a $3$-coloring compatible with $G$.
During the graph-building algorithm we will maintain such an order of the non-active intervals and also that $G$ induces a union of paths on the active intervals, thus a compatible coloring of $G$ will exist. At the end we prove that such a compatible $3$-coloring of the final graph is necessarily a $2$-proper $3$-coloring at any time.
%We say that an interval $I$ depends on $J$ if the directed edge $JI$ is in $G$. For every $I$ there are at most $2$ such intervals by our assumption on $G$, if there is no such $J$ then we say that $I$ is {\em independent}. Given an order on the intervals such that any interval depends only on intervals earlier in this order, then we can obtain a compatible $3$-coloring by starting with any coloring of the independent intervals and then coloring the dependent ones in this order. After building the graph we prove that there exists such an ordering and thus a compatible $3$-coloring and finally we show that such a coloring is necessarily a $2$-proper $3$-coloring at any time.

%we can naturally extend this coloring to all the intervals such that the coloring is {\em compatible} with the rules, that is, $I$ gets a color different from the color of $J_1,J_2$ for all dependent triples.
%For a representation with directed acyclic graphs - showing more clearly the similarities with the previous proof - see the proof of Theorem \ref{time}.
In the proof we say that we are coloring an interval $I$ differently from (resp.\ same as) another interval $J$ when we add an edge $IJ$ to $G$ with label different (resp.\ same).
Now we state the required properties.

\begin{enumerate}
\item Every point of the line is covered by at most two active intervals.
\item No active interval contains another active interval.
\item For any point on the line (at least) one of the following holds.
\begin{itemize}
\item[(a)] The point is contained in the same number of active intervals as original intervals, and additionally the labelling forces these original
intervals and these active intervals to have the same set of colors.
\item[(b)]
The labelling forces that the point is contained in original intervals of different colors.\end{itemize}
\item The order $<_a$ of the non-active intervals is appropriate. Also, the family of active intervals induces a union of paths in $G$, more precisely, two active intervals are connected and the connecting edge has label different if and only if they are consecutive in a chain.
%Every active interval depends on at most one interval, from which it has different color, i.e. it has at most one in-going edge labelled different.
\end{enumerate}

Note that the first two properties ensure that, just like in the proof of Theorem \ref{2col}, at any step the active intervals have a unique partition into disjoint (maximal) chains, thus property 4 is well-defined. 

Unlike in the previous proof, now points covered by only two active intervals are also important to us, so property 4 ensures that any two intersecting active intervals receive a different color.
In particular, property 3(b) must hold for every point contained in at least $2$ original intervals.
Apart from this, the arguments are similar to the ones in the previous proof.

%the following structure on the family of active intervals. Define a {\em chain} as a sequence of active intervals such that everyone intersects the one before and after it in the chain. The family of active intervals can be partitioned into disjoint chains. The last property guarantees that any coloring of the active intervals extends naturally and uniquely to a coloring of all the intervals which is compatible with the rules. dom: EZ NEM IGAZ, MEIRT UNIQUE?

%We will define the rules such that if we start by a proper coloring of the active intervals, then the extension is a quasi-online coloring (as required by the theorem) of the original family of intervals. Note that in the previous proof we started with an arbitrary coloring of the active intervals, which was not necessarily proper, thus now we additionally have to take care that a proper coloring of the active intervals extends to a coloring which is a proper coloring of the active intervals at any previous time as well. 

Now we define the graph-building algorithm. During the process, when an interval is deactivated, in $<_a$ it is placed after all previously deactivated intervals, that is, at the top of the order $<_a$. Also, during the process we never add nor delete an edge incident to any interval that was deactivated earlier, this way all previously deactivated intervals will necessarily remain appropriate. Note that we can and will indeed delete edges sometimes, in which case at the same time we add some edges that force that a compatible coloring with the new graph is necessarily also compatible with the deleted edge, thus the deleted edges are redundant. We do this solely to simplify our presentation.

In the first step we add $I_1$ to the (previously empty) graph and activate it.
In the inductive step we add $I_t$ to the graph and also to the family of active intervals.
If $I_t$ is covered by an active interval or by the union of two (consecutive) intervals of a chain, then
we deactivate $I_t$ and color it differently from the interval(s) covering it, i.e., we add at most two edges from $I_t$ to active intervals, labelled different. This way $I_t$ will be appropriate.
%Otherwise, if $I_t$ does not create a triple intersection among active intervals, it remains activated.

Properties 2 and 4 remain true. For a point $p$ if property 3(a) did hold before adding $I_t$, then if it is contained in a just deactivated interval, then at this moment of our algorithm property 3(b) will hold for $p$, otherwise property 3(a) remains true (as the possible change in the set of active intervals containing $p$ is the addition of $I_t$, which is an original interval as well).

\begin{figure}
\begin{center}
\includegraphics[width=14cm]{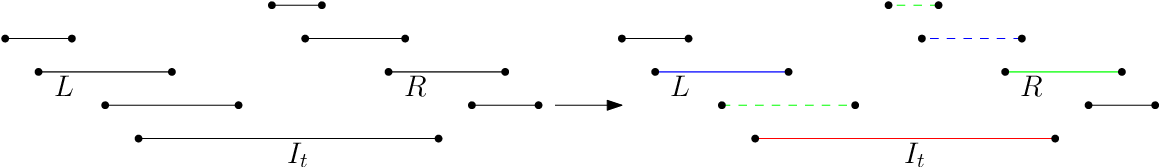}
\end{center}
\caption{Case i) of Theorem \ref{3col}}
\end{figure}

Again, the last thing we do in step $t$ of our graph-building algorithm is to ensure property 1, i.e., that no point is contained in three active intervals.
If there is no such point, we are done. Otherwise, such points must be in $I_t$. Denote (if exists) by $L$ the active interval with $l(L)<l(I_t)$ with the leftmost left end, and by $R$ the active interval with $r(I_t)<r(R)$ with the rightmost right end that covers a triple covered point. %dom: IDE NEM r KENE l HELYETT? ATIRTAM
They necessarily intersect $I_t$.
We distinguish two cases.\\

Case i) If $I_t$ is not covered by the union of the intervals in one chain, then either $L$ or $R$ does not exist,
or $L$ and $R$ are not in the same chain. In either case, $I_t$ is not covered by $L\cup R$.
We deactivate all active intervals covered by $L\cup I_t\cup R$, except for $L$, $I_t$ and $R$. The rule to color the now deactivated intervals is that they get a color different from $I_t$, i.e., for each deactivated interval $J$ we add an edge to $I_t$ labelled different. The order in which we deactivate these intervals (and add them to the top of $<_a$) is that one-by-one for each chain involved (the chain of $L$, $R$ and the chains in-between) we add first a left-most or right-most interval (in the chain containing $L$ (resp.\ $R$) it must be the rightmost (resp.\ leftmost)) and then one-by-one its neighbors. Adding the deactivated intervals in this order to $<_a$ ensures that by property 4 such a deactivated interval has at most two forward going edges (one to $I_t$ and at most one to one of its neighbors in the path corresponding to this chain).

%It is easy to check that the four properties are maintained.
%dom: AZT GEZA MIATT IRTAD BE, AMI EZUTAN JOTT ES KIKOMMENTELTEM? SZTEM TELJESEN FELESLEGES, VAGY AZ ELEJERE KELL ILYESMIKET IRNI, VAGY SEHOVA.\\

%Given a proper coloring of the active intervals at step $n$, by our rules it extends to a proper coloring of the active intervals in the previous step. Thus by induction at any time $t<n$ for any point $v$ it is either covered by differently colored intervals or it is covered by at most one interval. For time $n$ it is either covered by differently colored intervals or it is covered by as many original intervals as active intervals, and they have the same set of colors (by the third property). Now $v$ is either covered by two active intervals and two original intervals, then as the coloring was proper on the active intervals, the two active intervals and thus by the third property also the two original intervals, have different colors; or $v$ is covered by at most one active and thus by at most one original interval.\\

In the remaining cases $L$ and $R$ exist and are in one chain and $I_t$ is covered by the union of the intervals in this chain. Denote by $L_-$ and $L_+$ (resp.\ $R_-$ and $R_+$) the intervals preceding and succeeding $L$ (resp.\ $R$) in the chain (if they exist). 
%We deactivate all intervals covered by $L\cup I_t \cup R$ (including $I_t$), except for $L$ and $R$. Notice that apart from $I_t$ these intervals are all between $L$ and $R$ in this chain.

Case ii)
First assume that there are even many intervals in the chain between $L$ and $R$. %we deactivated an odd number of intervals this way (so an even number from the chain).
We insert a new active interval $N$ that we get by taking the union of $L$ and these intervals. We connect $N$ to $L_-$ (if it exists) and to $R$ with edges labelled different. Now we deactivate $L$ and color it the same as we color $N$ (we add the edge $NL$ to the graph labelled same and delete the other at most two edges from $L$ to $L_-$ and $L_+$).
We deactivate the intervals between $L$ and $R$ in the chain and color them differently from $I_t$. We deactivate them in the left-to-right order, thus they will be appropriate in $<_a$ (they have at most two forward edges, one to $I_t$ and one to their right neighbor in the chain). We deactivate $I_t$ and color differently from the color of $N$ and $R$ (we add the edges $I_tN$ and $I_tR$ to the graph). We need to check that the deleted edges became redundant. In a compatible coloring $L$ has the same color as $N$, different from $L_-$, as required, and also different from $R$, which must get the same color as $L_+$ (this is forced by a chain (a path in $G$) of active intervals all colored differently from $I_t$ and thus alternating), as required.

Case iii)
Next assume that there are odd many intervals between $L$ and $R$.
%we deactivated an even number of intervals this way (so an odd number from the chain).
We insert the new active interval $N=L\cup I_t\cup R$ and connect it to $L_-$ and $R_+$ (if they exist) with edges labelled different. Now we again deactivate $L$ and color it the same as we color $N$ (we add the edge $NL$ to the graph labelled same and delete the other at most two edges from $L$ to $L_-$ and $L_+$). We also deactivate $R$ and color it the same as we color $N$ (we add the edge $NR$ to the graph labelled same and delete the edge from $R$ to $R_+$; note that we do not delete the edge $RR_-$).
 %and deactivate all intervals contained in it, including $L$ and $R$.
%As before, we color the deactivated $I_t$ differently from the color of $N$ and the deactivated intervals of the chain in an alternating way differently from $I_t$.
We deactivate the intervals in the chain between $L$ and $R$ and color them differently from $I_t$. We deactivate them in the left-to-right order, thus they will be appropriate in $<_a$ (they have at most two forward edges, one to $I_t$ and one to their right neighbor in the chain). We deactivate $I_t$ and color differently from the color of $N$ (we add the edge $I_tN$ to the graph). We need to check that the deleted edges became redundant. In a compatible coloring $L$ has the same color as $N$, different from $L_-$, as required; $R$ has the same color as $N$, different from $R_+$, as required; 
the deactivated intervals from $L+$ to $R$ get the two colors different from the
color of $I_t$ alternatingly (forced by a chain which is a path in the graph), which ensures that $L$ and $L_+$ have different colors, as required.

\begin{figure}
\begin{center}
\includegraphics[width=14cm]{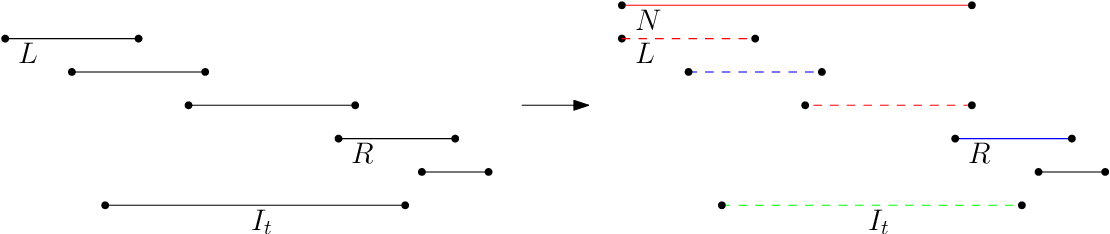}
\end{center}
\caption{Case ii) of Theorem \ref{3col}}
\end{figure}

\begin{figure}
\begin{center}
\includegraphics[width=14cm]{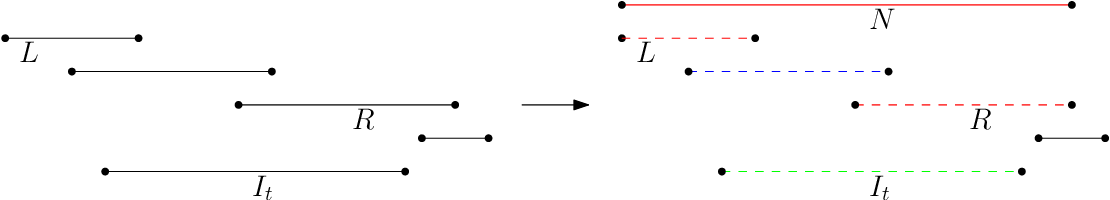}
\end{center}
\caption{Case iii) of Theorem \ref{3col}}
\label{fig:3coleven}
\end{figure}

This way we made sure that property 1 holds and it is easy to check that properties 2 and 4 remain true.
It is again easy to check that property 3 also remains true for every point.

By property 4 at the end the active intervals induce a family of paths, which can be easily colored properly (even with $2$ colors). As the ordering $<_a$ is appropriate, coloring the non-active intervals in backwards order the coloring extends to a coloring of all the intervals compatible with the labelling of the graph. We have to prove that for this coloring at any time $t$ any point $p$ contained by at least $2$ of the original intervals is non-monochromatic. During the process we just added edges and deleted redundant edges, thus a coloring compatible with the final graph is also compatible with the graph at time $t$. Thus at time $t$, by property 1 we have that such a $p$ is contained in at most two active intervals. If it is contained in exactly two active intervals then by property 4 it is contained in intervals of both colors. Otherwise $p$ is contained in at most one active interval but at least two original intervals, thus property 3(a) cannot hold. Thus property 3(b) must hold, which is exactly that $p$ is contained in original intervals of both colors.
%Also, similarly to the previous case, it can be easily checked that if we extend a proper coloring of the active intervals, then for its extension it is true at any time (for time $t<n$ by induction, otherwise by the way we defined the rules) that every point is either covered by at most one original interval or it is covered by intervals of different colors.
\end{proof}

\begin{theorem}\label{time} Colorings guaranteed by Theorem \ref{2col} and
Theorem \ref{3col} can be found in $O(n\log n)$ time.
\end{theorem}

\begin{proof}
Instead of a rigorous proof we provide only a sketch, the easy details are left
to the reader. In both algorithms we have $n$ intervals, thus $n$ steps. In each
step we define a bounded number of new active intervals, thus altogether we have
$cn$ original and auxiliary intervals. We always maintain the (well-defined)
left-to-right order of the active intervals. We also maintain an order of the
(original and auxiliary) intervals such that an interval's color depends only on the
color of one or two intervals' that are later in this order. This order can be
easily maintained as in each step the new interval and the new active intervals
come at the end of the order. We also save for each interval the one or two
intervals which it depends on. This can be imagined as the intervals represented by vertices on the
horizontal line arranged according to this order and an acyclic directed graph on them representing the dependency relations, thus each edge goes
backwards and each vertex has in-degree at most two (at most one in the first
algorithm, i.e., the graph is a directed forest in that case). In each step we
have to update the order of active intervals and the acyclic graph of all the
intervals, this can be done in $c \log n$ time plus the time needed for the
deletion of intervals from the order. Although the latter can be linear in one
step, yet altogether during the whole process it remains $cn$.
At the end we just color the vertices one by one from right to left
following the rules, which again takes only $cn$ time. Altogether this is
$cn\log n$ time.
\end{proof}

As we noted earlier, these problems are equivalent to (offline) colorings of bottomless rectangles in
the plane. Using this phrasing, Theorem \ref{3col} and Theorem \ref{2col} were
proved already in \cite{wcf,keszegh}, yet those proofs are quite
involved and they only give quadratic time algorithms, so these results are
improvements regarding simplicity of proofs and efficiency of the algorithms.
The algorithms in \cite{wcf,keszegh} proceed with the
intervals in backwards order and the intervals are colored immediately, yet in each
step many intervals have to be recolored, this might be a reason why a lot of
re-colorings are needed there (which we do not need in the above proofs), adding
up to quadratic time algorithms (contrasting the near-linear time algorithms
above).

%\bigskip
\section*{Acknowledgments}
This paper is an extended version of the conference publication \cite{onlinepaper}.  Part of this work was funded by the Department of Energy at Los Alamos National Laboratory under contract DE-AC52-06NA25396, and the DOE Office of Science Advanced Computing Research (ASCR) program in Applied Mathematics. We are thankful to an unanonymous reviewer for his many helpful comments.

%---------------------------- Bibliography -------------------------------

\bibliographystyle{abbrv}

%\pagebreak
\end{document}